\documentclass[12pt]{amsart}
\usepackage{amsmath,amssymb,amsbsy,amsfonts,latexsym,amsopn,amstext,cite,
                                               amsxtra,euscript,amscd,bm}
\usepackage{url}

\usepackage{mathrsfs}

\usepackage{color}
\usepackage[colorlinks,linkcolor=blue,anchorcolor=blue,citecolor=blue,backref=page]{hyperref}
\usepackage{color}
\usepackage{graphics,epsfig}
\usepackage{graphicx}
\usepackage{float}
\usepackage{epstopdf}
\hypersetup{breaklinks=true}

\usepackage[np]{numprint}
\npdecimalsign{\ensuremath{.}}

\usepackage{bibentry}

\usepackage[english]{babel}
\usepackage{mathtools}
\usepackage{todonotes}

\usepackage[colorlinks,linkcolor=blue,anchorcolor=blue,citecolor=blue,backref=page]{hyperref}

\usepackage[norefs,nocites]{refcheck}

\begin{document}

\newcommand{\bs}{\boldsymbol}
\def \a{\alpha} \def \b{\beta} \def \d{\delta} \def \e{\varepsilon} \def \g{\gamma} \def \k{\kappa} \def \l{\lambda} \def \s{\sigma} \def \t{\theta} \def \z{\zeta}

\newcommand{\mb}{\mathbb}

\newtheorem{theorem}{Theorem}
\newtheorem{lemma}[theorem]{Lemma}
\newtheorem{claim}[theorem]{Claim}
\newtheorem{cor}[theorem]{Corollary}
\newtheorem{conj}[theorem]{Conjecture}
\newtheorem{prop}[theorem]{Proposition}
\newtheorem{definition}[theorem]{Definition}
\newtheorem{question}[theorem]{Question}
\newtheorem{example}[theorem]{Example}
\newcommand{\hh}{{{\mathrm h}}}
\newtheorem{remark}[theorem]{Remark}
\newtheorem{prob}[theorem]{Problem}

\numberwithin{equation}{section}
\numberwithin{theorem}{section}
\numberwithin{table}{section}
\numberwithin{figure}{section}

\def\sssum{\mathop{\sum\!\sum\!\sum}}
\def\ssum{\mathop{\sum\ldots \sum}}
\def\iint{\mathop{\int\ldots \int}}

\newcommand{\diam}{\operatorname{diam}}
\newcommand{\chr}{\operatorname{char}}

\def\squareforqed{\hbox{\rlap{$\sqcap$}$\sqcup$}}
\def\qed{\ifmmode\squareforqed\else{\unskip\nobreak\hfil
\penalty50\hskip1em \nobreak\hfil\squareforqed
\parfillskip=0pt\finalhyphendemerits=0\endgraf}\fi}

\newfont{\teneufm}{eufm10}
\newfont{\seveneufm}{eufm7}
\newfont{\fiveeufm}{eufm5}
%
%
\newfam\eufmfam
     \textfont\eufmfam=\teneufm
\scriptfont\eufmfam=\seveneufm
     \scriptscriptfont\eufmfam=\fiveeufm
%
%
\def\frak#1{{\fam\eufmfam\relax#1}}

\newcommand{\bflambda}{{\boldsymbol{\lambda}}}
\newcommand{\bfmu}{{\boldsymbol{\mu}}}
\newcommand{\bfxi}{{\boldsymbol{\eta}}}
\newcommand{\bfrho}{{\boldsymbol{\rho}}}

\def\eps{\varepsilon}

\def\fK{\mathfrak K}
\def\fT{\mathfrak{T}}
\def\fL{\mathfrak L}
\def\fR{\mathfrak R}
\def\fQ{\mathfrak Q}

\def\fA{{\mathfrak A}}
 \def\fB{{\mathfrak B}}
\def\fC{{\mathfrak C}}
\def\fL{{\mathfrak L}}
\def\fM{{\mathfrak M}}
\def\fS{{\mathfrak  S}}
\def\fU{{\mathfrak U}}

\def\sssum{\mathop{\sum\!\sum\!\sum}}
\def\ssum{\mathop{\sum\ldots \sum}}
\def\dsum{\mathop{\quad \sum \qquad \sum}}
\def\iint{\mathop{\int\ldots \int}}
 
\def\T {\mathsf {T}}
\def\Tor{\mathsf{T}_d}
\def\Tore{\widetilde{\mathrm{T}}_{d} }

\def\sM {\mathsf {M}}
\def\sL {\mathsf {L}}
\def\sK {\mathsf {K}}
\def\sP {\mathsf {P}}

\def\ss{\mathsf {s}}

\def \balpha{\bm{\alpha}}
\def \bbeta{\bm{\beta}}
\def \bgamma{\bm{\gamma}}
\def \bdelta{\bm{\delta}}
\def \bzeta{\bm{\zeta}}
\def \blambda{\bm{\lambda}}
\def \bchi{\bm{\chi}}
\def \bphi{\bm{\varphi}}
\def \bpsi{\bm{\psi}}
\def \bxi{\bm{\xi}}
\def \bmu{\bm{\mu}}
\def \bnu{\bm{\nu}}
\def \bomega{\bm{\omega}}

\def \bell{\bm{\ell}}

\def\eqref#1{(\ref{#1})}

\def\vec#1{\mathbf{#1}}

\newcommand{\abs}[1]{\left| #1 \right|}

\def\Zq{\mathbb{Z}_q}
\def\Zqx{\mathbb{Z}_q^*}
\def\Zd{\mathbb{Z}_d}
\def\Zdx{\mathbb{Z}_d^*}
\def\Zf{\mathbb{Z}_f}
\def\Zfx{\mathbb{Z}_f^*}
\def\Zp{\mathbb{Z}_p}
\def\Zpx{\mathbb{Z}_p^*}
\def\cM{\mathcal M}
\def\cE{\mathcal E}
\def\cH{\mathcal H}

\def\le{\leqslant}
\def\leq{\leqslant}
\def\ge{\geqslant}
\def\leq{\leqslant}

\def\sfB{\mathsf {B}}
\def\sfC{\mathsf {C}}
\def\sfS{\mathsf {S}}
\def\sfI{\mathsf {I}}
\def\sfT{\mathsf {T}}
\def\L{\mathsf {L}}
\def\FF{\mathsf {F}}

\def\sB {\mathscr{B}}
\def\sE {\mathscr{E}}
\def\sS {\mathscr{S}}

\def\cA{{\mathcal A}}
\def\cB{{\mathcal B}}
\def\cC{{\mathcal C}}
\def\cD{{\mathcal D}}
\def\cE{{\mathcal E}}
\def\cF{{\mathcal F}}
\def\cG{{\mathcal G}}
\def\cH{{\mathcal H}}
\def\cI{{\mathcal I}}
\def\cJ{{\mathcal J}}
\def\cK{{\mathcal K}}
\def\cL{{\mathcal L}}
\def\cM{{\mathcal M}}
\def\cN{{\mathcal N}}
\def\cO{{\mathcal O}}
\def\cP{{\mathcal P}}
\def\cQ{{\mathcal Q}}
\def\cR{{\mathcal R}}
\def\cS{{\mathcal S}}
\def\cT{{\mathcal T}}
\def\cU{{\mathcal U}}
\def\cV{{\mathcal V}}
\def\cW{{\mathcal W}}
\def\cX{{\mathcal X}}
\def\cY{{\mathcal Y}}
\def\cZ{{\mathcal Z}}
\newcommand{\rmod}[1]{\: \mbox{mod} \: #1}

\def\cg{{\mathcal g}}

\def\vX{\mathbf X}
\def\vY{\mathbf Y}

\def\vy{\mathbf y}
\def\vr{\mathbf r}
\def\vx{\mathbf x}
\def\va{\mathbf a}
\def\vb{\mathbf b}
\def\vc{\mathbf c}
\def\vd{\mathbf d}
\def\ve{\mathbf e}
\def\vf{\mathbf f}
\def\vg{\mathbf g}
\def\vh{\mathbf h}
\def\vk{\mathbf k}
\def\vm{\mathbf m}
\def\vz{\mathbf z}
\def\vu{\mathbf u}
\def\vv{\mathbf v}

\def\e{{\mathbf{\,e}}}
\def\ep{{\mathbf{\,e}}_p}
\def\eq{{\mathbf{\,e}}_q}
\def\er{{\mathbf{\,e}}_r}
\def\es{{\mathbf{\,e}}_s}

 \def\SS{{\mathbf{S}}}

 \def\0{{\mathbf{0}}}
 
 \newcommand{\GL}{\operatorname{GL}}
\newcommand{\SL}{\operatorname{SL}}
\newcommand{\lcm}{\operatorname{lcm}}
\newcommand{\ord}{\operatorname{ord}}
\newcommand{\Tr}{\operatorname{Tr}}
\newcommand{\Span}{\operatorname{Span}}

\def\({\left(}
\def\){\right)}
\def\l|{\left|}
\def\r|{\right|}
\def\fl#1{\left\lfloor#1\right\rfloor}
\def\rf#1{\left\lceil#1\right\rceil}
\def\sumstar#1{\mathop{\sum\vphantom|^{\!\!*}\,}_{#1}}

\def\mand{\qquad \mbox{and} \qquad}

\def\tblue#1{\begin{color}{blue}{{#1}}\end{color}}




\hyphenation{re-pub-lished}

\mathsurround=1pt

\def\bfdefault{b}

\def \F{{\mathbb F}}
\def \K{{\mathbb K}}
\def \N{{\mathbb N}}
\def \Z{{\mathbb Z}}
\def \P{{\mathbb P}}
\def \Q{{\mathbb Q}}
\def \R{{\mathbb R}}
\def \C{{\mathbb C}}
\def\Fp{\F_p}
\def \fp{\Fp^*}

\def\PER{{\mathcal{PER}}}
\def\per{{\mathrm {per}\,}}
\def\imm{{\mathrm {imm}\,}}

 \def \xbar{\overline x}
 \def \Kbar{\overline K}
  \def \Fbar{\overline \F_p}
   \def \Qbar{\overline \Q}

\title[Rational numbers in short intervals]{Rational numbers with small denominators in short intervals}

\author[I. E. Shparlinski] {Igor E. Shparlinski}
\address{School of Mathematics and Statistics, University of New South Wales, Sydney NSW 2052, Australia}
\email{igor.shparlinski@unsw.edu.au}

\begin{abstract}  
We use bounds on bilinear forms with Kloosterman fractions and 
improve the error term in the asymptotic formula of  Balazard and Martin~(2023) 
on the average value of the smallest denominators of rational numbers  in short intervals.
\end{abstract}

\subjclass[2020]{11B57, 11L07}

\keywords{Farey fractions, short intervals, Kloosterman fractions}

\maketitle

%

\section{Introduction} 

Given integer $N \ge 1$,  and $j =1, \ldots, N$, we denote by $q_j(N)$ the smallest integer $q$ such that for some $a$ we have 
$$
\frac{a}{q} \in \( \frac{j-1}{N},  \frac{j}{N}\right]. 
$$

Next, we consider the average value 
$$
S(N) = \frac{1}{N} \sum_{j=1}^N q_j(N). 
$$

Recently, Balazard and Martin~\cite{BaMa} have confirmed the conjecture 
of Kruyswijk and Meijer~\cite{KrMe}  that 
$$
S(N) \sim  \frac{16}{\pi^2} N^{3/2}
$$
and in fact established  the following much more precise asymptotic formula
 \begin{equation}
\label{eq:S asymp}
S(N) = \frac{16}{\pi^2} N^{3/2}  + O\( N^{4/3} (\log N)^2\), 
\end{equation}
see~\cite[Equation~(1)]{BaMa}. Note that the asymptotic formula~\eqref{eq:S asymp} improves on previous upper and lower bounds of Kruyswijk and Meijer~\cite{KrMe} and Stewart~\cite{Ste}, for example on the previous inequalities 
$$
1.35N^{3/2} < S(N) < 2.04 N^{3/2} 
$$
in~\cite{Ste} (note that $16/\pi^2 = 1.6211 \ldots$). 
For other related results, see~\cite{Art, BSY, ChHa, E-BLS, Mark1, Mark2} and references therein.

The bound on the error term in~\eqref{eq:S asymp} is based on the classical bound
of Kloosterman sums, see, for example,~\cite[Corollary~11.12]{IwKow}.

Here, we use bounds on bilinear forms with Kloosterman fractions due to 
Duke, Friedlander and Iwaniec~\cite{DFI}  and 
improve the error term in the asymptotic formula~\eqref{eq:S asymp} as follows.

\begin{theorem}\label{thm: S(N)}
We have
$$
S(N) = \frac{16}{\pi^2} N^{3/2}  + O\(  N^{29/22 + o(1)} \), 
$$
as $N\to  \infty$. 
\end{theorem}

\section{Preliminary reductions} 
\label{sec:prelim} 

As usual, we use the expressions $U \ll V$ and $U=O(V)$ to
mean $|U|\leq c V$ for some constant $c>0$ which throughout this paper
is absolute. 

 We have 
 \begin{equation}
\label{eq:S and R}
S(N) = \frac{16}{\pi^2} N^{3/2}  + R(N), 
\end{equation} 
where by~\cite[Equations~(19), (20) and~(21)]{BaMa} we can write 
 \begin{equation}
\label{eq:R and T}
R(N)  \ll T_{11}(N) +  T_{12}(N) + T_{2}(N) 
\end{equation} 
for some quantities $T_{11}(N)$, $T_{12}(N)$ and $T_{2}(N)$ which are estimated in~\cite{BaMa} separately.  In particular, by~\cite[Equations~(23) and~(26)]{BaMa} 
we have
 \begin{equation}
\label{eq:T12 and T2}
T_{12}(N)  \ll  N^{5/4} (\log N)^2 \mand T_{2}(N)    \ll  N^{5/4} (\log N)^2.
\end{equation} 
Therefore. the error term in~\eqref{eq:S asymp}  comes from the bound 
  \begin{equation}
\label{eq:T11}
T_{11}(N)  \ll  N^{4/3} (\log N)^2 
\end{equation} 
given by~\cite[Equation~(22)]{BaMa}. 
 
 We  now see from~\eqref{eq:S and R},  \eqref{eq:R and T}  and~\eqref{eq:T12 and T2}, 
 that in order to   establish Theorem~\ref{thm: S(N)}
 we only need to improve~\eqref{eq:T11} as
   \begin{equation}
\label{eq:T11-new}
T_{11}(N)  \ll  N^{29/22 + o(1)} . 
\end{equation}
 
 We first recall the following expression for  $T_{11}(N)$ given in~\cite[Section~5.3]{BaMa}:
 \begin{equation}
\label{eq:T11-B}
T_{11}(N) = \sum_{s \ge \sqrt{N} } \sum_{\substack{1 \le r \le R_s\\\gcd(r,s)=1}}
r B_1\(\frac{N r^{-1}}{s}\)
\end{equation} 
with the Bernoulli function
$$
B_1(u) = \begin{cases}
0, & \text{if}\ u \in \Z, \\
 \{u\}-1/2, & \text{if}\ u \in \Z, 
 \end{cases}
 $$
 where $\{u\}$ is the fractional part of a real $u$,  
 the inversion $ r^{-1}$ in the fractional part  $\{N r^{-1}/s\}$ is computed modulo $s$
  and $R_s$ is a certain sequence of positive integers, satisfying 
  \begin{equation}
\label{eq:Rs}
R_s \ll N/s
\end{equation} 
(we refer to~\cite{BaMa} for an exact definition, which is not important for our argument).

It is more convenient for us to work with the function 
$$
\psi(u) = \{u\}-1/2,
$$ 
which coincides with $B_1(u) $ for all $u \not \in \Z$. 

In particular,   
$$
B_1\(\frac{N r^{-1}}{s}\)= \psi\(\frac{N r^{-1}}{s}\)
$$
unless $s \mid N$. 

Using the classical bound on the divisor function
 \begin{equation}
\label{eq:Div}
\tau(k) = k^{o(1)},
\end{equation} 
for an integer positive $k \to \infty$
(see, for example,~\cite[Equation~(1.81)]{IwKow}), 
we infer from~\eqref{eq:T11-B} that  
 \begin{equation}
\label{eq:T11-psi}
T_{11}(N) =U(N) +  E(N) ,
\end{equation} 
where 
 \begin{equation}
\label{eq:UN}
U(N) = \sum_{s \ge \sqrt{N} } \sum_{\substack{1 \le r \le R_s\\\gcd(r,s)=1}}
r \psi\(\frac{N r^{-1}}{s}\), 
\end{equation} 
and, using~\eqref{eq:Rs}, 
 \begin{equation}
\label{eq:EN}
E(N) \ll \sum_{\substack{s \ge  \sqrt{N} \\s \mid N}} R_s^2 
\ll N^2   \sum_{\substack{s \ge  \sqrt{N} \\s \mid N}} s^{-2}  \le N^{1+o(1)}.
\end{equation}

\section{Vaaler polynomials} 

By a result of Vaaler~\cite{Vaal},
see also~\cite[Theorem~A.6]{GrKol} we have the following approximation to $\psi(u)$.

\begin{lemma}
\label{lem:Vaal Approx}
For any integer $H\ge 1$ there is a trigonometric 
polynomial
$$
\psi_H(u)  = \sum_{1\leq \abs{h} \leq H} \frac{a_h}{-2i\pi h} \e(hu)
$$
for coefficients $a_h\in[0,1]$ and such 
that 
 \begin{equation}
\label{eq:Vaal Ined}
\abs{\psi(u)-\psi_H(u)}
\le \frac{1}{2H+2} \sum_{\abs{h}\leq H} \(1-\frac{\abs{h}}{H+1}\) \e(hu).
\end{equation}
\end{lemma}

\section{Bilinear forms with Kloosterman fractions} 

For an integer $q$, let  $\e(z) = \exp(2 \pi z)$. 
Here we collect some estimates on bilinear form with 
exponantials 
$ \e\(hr^{-1}/s\)$ where,  as before,  $ r^{-1}$ in the argument is computed modulo $s$. 

For $U\ge 1$ we  aslo also use $u \sim U$ to indicate $U \le u < 2U$. 

We start with recalling the following bound  of Duke, Friedlander and Iwaniec~\cite[Theorem~1]{DFI}. 

\begin{lemma}
\label{lem:DFI} For sequences $\balpha = \{\alpha_r\}_{r=1}^\infty$, $\bbeta = \{\beta_s\}_{s=1}^\infty$ of complex numbers, an nonzero integer $K $ and real positive 
$R$ and $S$ 
we have
\begin{align*}
& \abs{\sum_{s \sim S}  \sum_{\substack{r \sim R\\\gcd(r,s) = 1}} 
 \alpha_r \beta_s \e\(Kr^{-1}/s\)}\\
 &\qquad \quad   \le  \|\balpha \| \|\bbeta \|   \((R + S)^{1/2} +\(1 + \frac{K}{RS}\)^{1/2} \min\{R,S\}\) (RS)^{o(1)}, 
\end{align*}
where 
$$
 \|\balpha \| =\( \sum_{r \sim R}|\alpha_r|^2\)^{1/2} \mand  \|\bbeta \| =\( \sum_{s \sim S}|\beta_s|^2\)^{1/2}.
$$
\end{lemma}

Next,  given  two sequences  of complex numbers
$$
\balpha = \{\alpha_r\}_{r=1}^\infty \mand \bbeta = \{\beta_s\}_{s=1}^\infty, 
$$
a sequence of positive integers 
$$
\cR =  \{\beta_s\}_{s=1}^\infty
$$
and an integer $h$, for $S \ge 1$ we define the bilinear form
$$
\sB_K(S; \cR,  \balpha,  \bbeta) = \sum_{s \sim S}  \sum_{\substack{r =1\\\gcd(r,s) = 1}}^{R_s}
 \alpha_r \beta_s \e\(Kr^{-1}/s\).
$$
Note that in the sums $\sB_K(S; \cR,  \balpha,  \bbeta)$ the range of summation over $r$ depends on $s$
and hence Lemma~\ref{lem:DFI} does not directly apply. 

We observe that for 
 \begin{equation}
\label{eq:Special case}
\alpha_r = r, \quad \beta_s \ll 1, \quad  R_s \ll \min\{N/s, s\},   \qquad r,s =1, 2, \ldots, 
\end{equation}
the argument in~\cite[Section~3]{BaMa} (in which we also inject the bound~\eqref{eq:Div})  immediately implies that for 
$$
0 < |K| = N^{O(1)} \mand 0 < S \ll N
$$
we have
 \begin{equation}
\label{eq:Bound-Weil}
\begin{split}
\sB_K(S; \cR,  \balpha,  \bbeta) &  \ll \sum_{s \sim S} \gcd(K, s)^{1/2} R_s s^{1/2} \log s\\
& \le
N^{1+o(1)}\sum_{s \sim S} \gcd(K, s)^{1/2}   s^{-1/2} \\
& \le
N^{1+o(1)}  S^{-1/2}  \sum_{d\mid K} d^{1/2}  \sum_{\substack{s \le 2S\\ d \mid s}}   1   \\
& \le N^{1+o(1)}  S^{-1/2}  \sum_{d\mid K} d^{1/2} \fl{2S/d} \\
& \le N^{1+o(1)}  S^{1/2}  \sum_{d\mid K} d^{-1/2}  \\
& \le   N^{1+o(1)}  S^{1/2}.
\end{split}
\end{equation}
Note that one can also derive~\eqref{eq:Bound-Weil}   via~\cite[Lemma~8]{DFI} and partial summation. 

In fact using the bound~\eqref{eq:Bound-Weil} for $S \le N^{2/3}$ and the trivial bound  
$$
\sB_K(S; \cR,  \balpha,  \bbeta)  \ll \sum_{s \sim S} 
R_s^2 \ll N^2S^{-1}
$$
in our argument below, one  recovers  the asymptotic formula~\eqref{eq:S asymp}.
However using some other bounds we achieve a stronger result. 

We also remark that for us only the choice of 
$\balpha = \{\alpha_r\}_{r=1}^\infty$ satisfying~\eqref{eq:Special case}
matter. However we present the below results for a more general $\balpha$
(but still they admit even more general forms).

Using Lemma~\ref{lem:DFI} together with  the standard completing technique, see, for example,~\cite[Section~12.2]{IwKow}, we derive our main technical tool.

\begin{lemma}
\label{lem:DFI-Gen} For sequences $\balpha = \{\alpha_r\}_{r=1}^\infty$, $\bbeta = \{\beta_s\}_{s=1}^\infty$ and  $\cR =  \{R_s\}_{s=1}^\infty$, an  nonzero integer $K$ and real $S$ with 
$$
\alpha_r  \ll A, \quad \beta_s \ll B, \quad  R_s \ll \min\{N/s, s\},   \qquad r,s =1, 2, \ldots, 
$$
and 
$$
 N^{1/2} \ll  S \ll N, 
$$
we have 
$$
\abs{\sB_K(S; \cR,  \balpha,  \bbeta)} \le AB (RS)^{1/2} \(S^{1/2} +  R+ K^{1/2} S^{-1/2} R^{1/2}\)N^{o(1)}, 
$$
where 
$$
R =  \max\{R_s: ~ s\sim S\}. 
$$
\end{lemma}

\begin{proof}  
Note that 
 \begin{equation}
\label{eq:R N S}
R \ll N/S \ll S.
\end{equation}

Using the orthogonality of exponential functions, we write 
\begin{align*}
\sB_K(S&; \cR,  \balpha,  \bbeta) \\ & =  \sum_{s \sim S}  \sum_{\substack{r =1\\\gcd(r,s) = 1}}^{R_s}
 \alpha_r \beta_s  \e\(Kr^{-1}/s\)\\
  & =   \sum_{s \sim S}  \sum_{\substack{r =1\\\gcd(r,s) = 1}}^{R}
 \alpha_r \beta_s  \e\(Kr^{-1}/s\) \frac{1}{R} \sum_{u=0}^{R-1}\sum_{t=1}^{R_s}
  \e(u(t - r)/R) \\
 & =  \frac{1}{R} \sum_{u=0}^{R-1}  \sum_{s \sim S}  \sum_{\substack{r =1\\\gcd(r,s) = 1}}^{R}
 \alpha_r    \e(-ur/R) \beta_s  \e\(Kr^{-1}/s\)  \sum_{t=1}^{R_s}  \e(ut/R) .
\end{align*}
Using that 
$$
\sum_{t=1}^{R_s}  \e(ut/R) 
  \ll \frac{R}{\min\{u, R- u\} + 1},
$$
see~\cite[Equation~(8.6)]{IwKow},  we derive 
\begin{align*}
\sB_K(S; \cR,  \balpha,  \bbeta)  \ll  \frac{1}{R} &
\sum_{u=0}^{R-1}  \frac{R}{\min\{u, R- u\} + 1}\\
&\quad  \times  \abs{ \sum_{s \sim S}  \sum_{\substack{r =1\\\gcd(r,s) = 1}}^{R}
 \alpha_r    \e(-ur/R) \beta_s  \e\(Kr^{-1}/s\) } .
 \end{align*}

 It remains to observe that  for each $u = 0, \ldots, R-1$ the bound of Lemma~\ref{lem:DFI} applies to the inner sum and implies 
\begin{align*}
& \abs{\sB_K(S; \cR,  \balpha,  \bbeta)}     \\
 & \qquad \quad \le  AB (RS)^{1/2} \((R + S)^{1/2} + \(1 + \frac{K}{RS}\)^{1/2} \min\{R,S\}\)N^{o(1)}.
 \end{align*}
Recalling~\eqref{eq:R N S}, this now simplifies as
\begin{align*}
\abs{\sB_K(S; \cR,  \balpha,  \bbeta)}    &\le  AB (RS)^{1/2} \(S^{1/2} +  \(1 + \frac{K}{RS}\)^{1/2}R\)N^{o(1)}\\
  & =  AB (RS)^{1/2} \(S^{1/2} +  R+ K^{1/2} S^{-1/2} R^{1/2}\)N^{o(1)}, 
 \end{align*}
 which concludes the proof. 
\end{proof}

\begin{remark} 
Instead of using Lemma~\ref{lem:DFI}, that is, essentially~\cite[Theorem~1]{DFI},  one can also derive a version of 
 Lemma~\ref{lem:DFI-Gen} from~\cite[Theorem~2]{DFI}, 
or from a stronger result due to  Bettin and  Chandee~\cite[Theorem~1]{BettChan}. However these bounds do not seem to improve our main result. 
\end{remark}

\section{Proof of Theorem~\ref{thm: S(N)}}

As we have noticed in Section~\ref{sec:prelim}, it is only enough to estimate  $T_{11}(N)$, as we borrow the bounds 
on  $T_{12}(N)$ and  $T_{2}(N)$ from~\cite{BaMa}. 
Furthermore, we see from~\eqref{eq:T11-psi} and~\eqref{eq:EN} that it is enough to estimate $U(N)$
given by~\eqref{eq:UN}.

We note that it is important to observe that the sum defining $\psi_H(u)$ in Lemma~\ref{lem:Vaal Approx} 
does not contain the term with $h=0$, while the sum on the right hand side of~\eqref{eq:Vaal Ined}
does. 
Hence, for any integer $H\ge 1$,  by  Lemma~\ref{lem:Vaal Approx} we have
\begin{align*}
 U(N)  &  \ll  H^{-1}  \sum_{s \ge \sqrt{N} }
  \sum_{\substack{r =1\\\gcd(r,s) = 1}}^{R_s} r  \\
  & \qquad +  \sum_{1\leq \abs{h} \leq H} \frac{1}{ h}\abs{  \sum_{s \ge \sqrt{N} }
  \sum_{\substack{r =1\\\gcd(r,s) = 1}}^{R_s} r   \e\(hNr^{-1}/s\) },\\
&  \qquad  \qquad  +   \frac{1}{ H}  \sum_{1\leq \abs{h} \leq H}\abs{   \sum_{s \ge \sqrt{N} }\sum_{\substack{r =1\\\gcd(r,s) = 1}}^{R_s}
 r   \e\(hNr^{-1}/s\) }\\
&  \ll  H^{-1}  \sum_{s \ge \sqrt{N} } R_s^2 +  \sum_{1\leq \abs{h} \leq H} \frac{1}{ h}\abs{  \sum_{s \ge \sqrt{N} }  \sum_{\substack{r =1\\\gcd(r,s) = 1}}^{R_s} r   \e\(hNr^{-1}/s\) }\\
&  \ll  H^{-1} N^{3/2} +  \sum_{1\leq \abs{h} \leq H} \frac{1}{ h}\abs{  \sum_{s \ge \sqrt{N} }  \sum_{\substack{r =1\\\gcd(r,s) = 1}}^{R_s} r   \e\(hNr^{-1}/s\) }.  
 \end{align*}
 
 Note that $R_s\ge 1$ implies $s \ll N$. 
Therefore,  partitioning the corresponding summation over $s$ into dyadic intervals, 
 we see that there is some integer $S$ with 
$$
 N^{1/2} \ll  S \ll N
 $$
and  such that 
 \begin{equation}
\label{eq:VUW}
U(N) \ll H^{-1}  N^{3/2} + V(N,S)  \log N,
\end{equation}
where 
$$
 V(N,S)  =   \sum_{1\leq \abs{h} \leq H} \frac{1}{ h}\abs{  \sum_{s \sim S} \sum_{\substack{r =1\\\gcd(r,s) = 1}}^{R_s}
 r   \e\(hNr^{-1}/s\) }.
$$
 
 Now, if $S \le H^{1/5} N^{3/5}$ then we use the bound~\eqref{eq:Bound-Weil} and easily derive
 \begin{equation}
\label{eq: Small S}
  V(N,S)  \le  N^{1+o(1)}  S^{1/2} \le H^{1/10} N^{13/10+ o(1)}.
\end{equation}
 
 On the other hand, for $S >  H^{1/5} N^{3/5}$, Lemma~\ref{lem:DFI-Gen} (used with $A\ll N/S$ and $B \ll 1$), after recalling that $R \ll N/S$, 
 implies the same bound:
\begin{align*}
  \sum_{s \sim S}  \sum_{\substack{r =1\\\gcd(r,s) = 1}}^{R_s}
 & r   \e\(hNr^{-1}/s\)\\
  & \le 
  (N/S)   N^{1/2+o(1)}  \(S^{1/2} +  NS^{-1}  +  h^{1/2} N S^{-1} \)\\
   & \le 
  (N/S)   N^{1/2+o(1)}  \(S^{1/2} +    h^{1/2} N S^{-1} \).
 \end{align*}
Therefore, recalling that  $S > H^{1/5} N^{3/5}$, we obtain 
\begin{align*}
  V(N,S) & \le (N/S) N^{1/2+o(1)}  \(S^{1/2} +    H^{1/2} N S^{-1} \)\\
  & =  N^{3/2+o(1)} S^{-1/2} +  H^{1/2}  N^{5/2+o(1)} S^{-2}\\
  & \le H^{-1/10} N^{6/5+o(1)}  +  H^{1/10} N^{13/10+o(1)} \\
     & \le H^{1/10} N^{13/10+o(1)} .
   \end{align*}
   
 Therefore, the bound~\eqref{eq: Small S} holds for any $S$. 
Substituting~\eqref{eq: Small S} in~\eqref{eq:VUW}  yields 
 $$
U(N) \ll H^{-1}  N^{3/2} +  H^{1/10} N^{13/10+o(1)}
$$
and choosing 
 $$
 H = \rf{N^{2/11}}
 $$
 to optimise the bound, we obtain 
 $$
U(N) \ll  N^{29/22+o(1)}.
$$

Finally, recalling~\eqref{eq:T11-psi} and~\eqref{eq:EN}, 
we derive~\eqref{eq:T11-new}  and conclude the proof.

\section*{Acknowledgements}

The author is very grateful to Michel Balazard, Bruno Martin and Jens Marklof  for encouragement 
and  very stimulating discussions and comments.  The author also would like to
thank the referee for a careful reading, which revealed 
some inaccuracies in the original version of the paper.

During the preparation of this work, the author was   supported in part by the  
Australian Research Council Grant DP230100534.


 \end{document}